\documentclass[11pt]{amsart}
\usepackage{amscd}
\usepackage{amsmath}
\usepackage{amsxtra}
\usepackage{amsfonts}
\usepackage{amssymb}

 \def\lcm{\operatorname*{lcm}}
\newtheorem{theorem}{Theorem}[section]
\newtheorem{corollary}[theorem]{Corollary}
\newtheorem{lemma}[theorem]{Lemma}
\newtheorem{proposition}[theorem]{Proposition}

\newtheorem{conjecture}[theorem]{Conjecture}
\theoremstyle{definition}
\newtheorem{definition}[theorem]{Definition}
\newtheorem{remark}[theorem]{Remark}

\newtheorem*{question}{Question}
\newtheorem{example}[theorem]{Example}
\theoremstyle{remark}

\renewcommand{\theclaim}{\textup{\theclaim}}

\newtheorem*{acknowledgements}{Acknowledgements}

\numberwithin{equation}{section}

\def\openone

{\mathchoice

{\hbox{\upshape \small1\kern-3.3pt\normalsize1}}

{\hbox{\upshape \small1\kern-3.3pt\normalsize1}}

{\hbox{\upshape \tiny1\kern-2.3pt\SMALL1}}

{\hbox{\upshape \Tiny1\kern-2pt\tiny1}}}

\makeatletter

\newbox\ipbox

\newcommand{\diracb}[1]{\left\langle #1\mathrel{\mathchoice

{\setbox\ipbox=\hbox{$\displaystyle \left\langle\mathstrut
#1\right.$}

\vrule height\ht\ipbox width0.25pt depth\dp\ipbox}

{\setbox\ipbox=\hbox{$\textstyle \left\langle\mathstrut
#1\right.$}

\vrule height\ht\ipbox width0.25pt depth\dp\ipbox}

{\setbox\ipbox=\hbox{$\scriptstyle \left\langle\mathstrut
#1\right.$}

\vrule height\ht\ipbox width0.25pt depth\dp\ipbox}

{\setbox\ipbox=\hbox{$\scriptscriptstyle \left\langle\mathstrut
#1\right.$}

\vrule height\ht\ipbox width0.25pt depth\dp\ipbox}

}\right. }

\newcommand{\dirack}[1]{\left. \mathrel{\mathchoice

{\setbox\ipbox=\hbox{$\displaystyle \left.\mathstrut
#1\right\rangle$}

\vrule height\ht\ipbox width0.25pt depth\dp\ipbox}

{\setbox\ipbox=\hbox{$\textstyle \left.\mathstrut
#1\right\rangle$}

\vrule height\ht\ipbox width0.25pt depth\dp\ipbox}

{\setbox\ipbox=\hbox{$\scriptstyle \left.\mathstrut
#1\right\rangle$}

\vrule height\ht\ipbox width0.25pt depth\dp\ipbox}

{\setbox\ipbox=\hbox{$\scriptscriptstyle \left.\mathstrut
#1\right\rangle$}

\vrule height\ht\ipbox width0.25pt depth\dp\ipbox}

} #1\right\rangle}

\newcommand{\bz}{\mathbb{Z}}

\newcommand{\br}{\mathbb{R}}

\newcommand{\bn}{\mathbb{N}}

\newcommand{\beq}{\begin{equation}}
\newcommand{\eeq}{\end{equation}}

\def\blfootnote{\xdef\@thefnmark{}\@footnotetext}

\renewcommand{\mod}{\operatorname{mod}}

\input xy
\xyoption{all}
\usepackage{amssymb}

\def\-{^{-1}}

\begin{document}

\title[Number theory problems from the harmonic analysis of a fractal]{Number theory problems from the harmonic analysis of a fractal}
\author{Dorin Ervin Dutkay}

\address{[Dorin Ervin Dutkay] University of Central Florida\\
	Department of Mathematics\\
	4000 Central Florida Blvd.\\
	P.O. Box 161364\\
	Orlando, FL 32816-1364\\
U.S.A.\\} \email{Dorin.Dutkay@ucf.edu}

\author{John Haussermann}

\address{[John Haussermann] University of Central Florida\\
	Department of Mathematics\\
	4000 Central Florida Blvd.\\
	P.O. Box 161364\\
	Orlando, FL 32816-1364\\
U.S.A.\\} \email{jhaussermann@knights.ucf.edu}

\thanks{} 

\begin{abstract}
We study some number theory problems related to the harmonic analysis (Fourier bases) of the Cantor set introduced by Jorgensen and Pedersen in \cite{JP98}. 
 
\end{abstract}
\maketitle \tableofcontents

\section{Introduction}
In \cite{JP98}, Jorgensen and Pedersen made a surprising discovery: they constructed a fractal measure on a Cantor set which has an orthonormal Fourier series. This Cantor set is obtained from the interval $[0,1]$, dividing it into four equal intervals and keeping the first and the third, $[0,1/4]$ and $[1/2,3/4]$, and repeating the procedure infinitely many times. It can be described in terms of iterated function systems: let 
$$\tau_0(x)=x/4\mbox{ and }\tau_2(x)=(x+2)/4,\quad(x\in\br).$$
The Cantor set $X_4$ is the unique compact set that satisfies the invariance condition 
$$X_4=\tau_0(X_4)\cup \tau_2(X_4).$$
The set $X_4$ is described also in terms of the base 4 decomposition of real numbers :
$$X_4=\left\{\sum_{k=1}^n 4^{-k}b_k : b_k\in\{0,2\}, n\in\bn\right\}.$$

On the set $X_4$ one considers the Hausdorff measure $\mu$ of dimension $\log_42=\frac12$. In terms of iterated function systems, the measure $\mu$ is the invariant measure for the iterated function system, that is, the unique Borel probability measure that satisfies the invariance equation 
\begin{equation}
\mu(E)=\frac12\left(\mu(\tau_0^{-1}E)+\mu(\tau_2^{-1}E)\right),\mbox{ for all Borel sets }E\subset\br.
\label{eqi1.1}
\end{equation}
Equivalently, for all continuous compactly supported functions $f$,
\begin{equation}
\int f\,d\mu=\frac12\left(\int f\circ\tau_0\,d\mu+\int f\circ\tau_2\,d\mu\right).
\label{eqi1.2}
\end{equation}

We denote, for $\lambda\in\br$: 
$$e_\lambda(x)=e^{2\pi i\lambda\cdot x},\quad(x\in\br).$$
Jorgensen and Pedersen proved in that the Hilbert space $L^2(\mu)$ has an orthonormal basis formed with exponential functions, i.e., a Fourier basis, $E(\Gamma_0):=\{e_\lambda : \lambda\in\Gamma_0\}$ where 
\begin{equation}
\Gamma_0:=\left\{\sum_{k=0}^n4^kl_k : l_k\in\{0,1\},n\in\bn\right\}.
\label{eqi1.3}
\end{equation}

Later, Strichartz \cite{MR2279556} proved that these Fourier series have better convergence properties than their classical counterparts on the unit interval; for example, the Fourier series of a continuous function converge uniformly. 

\begin{definition}\label{def1.1}
We say that the subset $\Gamma$ of $\br$ is a {\it spectrum} for the measure $\mu$ if the corresponding family of exponential functions $E(\Gamma):=\{e_\lambda : \lambda\in\Gamma\}$ is an orthonormal basis for $L^2(\mu)$. We say that $\Gamma$ is complete/incomplete if the set $E(\Gamma)$ is as such in $L^2(\mu)$.
\end{definition}

Other spectra for the measure $\mu$ were constructed later in \cite{MR1929508,Str00,DJ06,DHS09,MR3055992}, using some other digits for the spectrum. As we can see in \eqref{eqi1.3}, the spectrum $\Gamma_0$ corresponds to the digits $\{0,1\}$. 

The main question that we address in this paper is the following:

\begin{question}
For what digits $\{0,m\}$ with $m\in\bn$ odd is the set 
$$\Gamma(m):=m\Gamma_0=\left\{\sum_{k=0}^n 4^kl_k : l_k\in\{0,m\}, n\in\bn\right\}$$
a spectrum for $L^2(\mu)$?
\end{question}

\begin{definition}\label{def1.2}
Let $m\in\bn$ be an odd number. We say that $m$ is {\it complete} if the set $\Gamma(m)$ is a spectrum for the measure $\mu$. We say that $m$ is {\it incomplete} if it is not complete. 
\end{definition}

As it was shown in \cite{DJ06}, that the set $E(\Gamma(m))$ is always orthonormal in $L^2(\mu)$, but sometimes it is incomplete. For example, for $m=3$, the set $\Gamma(3)$ is not complete. Applying the results from \cite{MR1929508} or the refinement obtained in \cite{DJ06}, we can characterize the numbers $m$ that give spectra (i.e., {\it complete} orthonormal bases) in terms of {\it extreme cycles}. 
\begin{definition}\label{defi1.2}
Let $m\in\bn$ be an odd number. 
We say that a finite set $\{x_0,x_1,\dots,x_{r-1}\}$ is an {\it extreme cycle} (for the digits $\{0,m\}$) if there exist $l_0,\dots, l_{r-1}\in\{0,m\}$ such that 
$$x_1=\frac{x_0+l_0}4,\quad x_2=\frac{x_1+l_1}4,\quad\dots\quad,x_{r-1}=\frac{x_{r-2}+l_{r-2}}4,\quad x_0=\frac{x_{r-1}+l_{r-1}}4,$$
and 
\begin{equation}
\left|\frac{1+e^{2\pi i 2x_k}}{2}\right|=1,\quad(k\in\{0,\dots,r-1\}).
\label{eqi1.2.1}
\end{equation}
The points $x_i$ are called extreme cycle points.
\end{definition}

\begin{theorem}\label{thi1.3}\cite{MR1929508,DJ06}
Let $m\in\bn$ be odd. The number $m$ is complete if and only if the only extreme cycle for the digit set $\{0,m\}$ is the trivial one $\{0\}$.
\end{theorem}

For example, for $m=3$, the set $\{1\}$ is an extreme cycle: $(1+3)/4=1$ and $e^{2\pi i2\cdot 1}=1$, so $\Gamma(3)$ is incomplete. 

In \cite{DJ10} it was proved that the sets $\Gamma(5^k)$ are complete for any $k$, which shows the surprising fact that spectra have arbitrarily low densities. In \cite{MR3055992} it was shown that there are spectra for this fractal measure which have zero Beurling dimension. The result from \cite{DJ10} was used by Jorgensen et al. to construct some scaling operators on the Cantor set, operators that exhibit an interesting fractal structure \cite{MR2966145,MR3202868}. 

Theorem \ref{thi1.3} turns our question into a number theory question: for what odd numbers $m$ are there no (non-trivial) extreme-cycles? Any odd number $m$ satisfying this criterion is complete; any odd number $m$ not satisfying this criterion is incomplete.

We show in Propostion \ref{pri1.5} that, if a number is incomplete, then all its multiples are incomplete. Because of this, we introduce a new notion of primitive numbers:

\begin{definition}\label{defprim}
We say that an odd number $m$ is {\it primitive} if $m$ is incomplete and, for all proper divisors $d$ of $m$, $d$ is complete. In other words, there exist non-trivial extreme cycles for the digits $\{0,m\}$ and there are no non-trivial extreme cycles for the digits $\{0,d\}$ for any proper divisor $d$ of $m$.
\end{definition} 

Of course, a number $m$ will be incomplete if and only if it is divisible by a primitive number. A computer check shows that the first primitive numbers are: 3, 85, 341,
 455,
 1285,
 4369,
 5461,
 6355,
 9709,
 28679,
 60787,
 327685,
 416179. See Table \ref{tab1} for more primitive numbers. So, in particular, the numbers $3k,85k,341k,455k,1285k$ etc. are incomplete for any odd natural number $k$.  The primitive numbers seem to become more and more sparse, but we prove in Theorem \ref{th1.10} that there are infinitely many primitive numbers.

\begin{table}[ht]\label{tab1}
{	
\begin{center}

\begin{tabular}{|c|c|c|}
\hline
$m$ & Prime decomposition & $o_4$ for the primes\\
\hline
3 & 3 & 1\\
\hline
85 & 5,17 & 2,4 \\
\hline
341& 11,31&5,5\\
\hline
455&	5,7,13&	2,3,6\\
\hline
1285&	5,257&	2,8\\
\hline
4369&	17,257&	4,8\\
\hline
5461&	43,127&	7,7\\
\hline
6355&	5,31,41&	2,5,10\\
\hline
9709&	7,19,73&	3,9,9\\
\hline
28679&	7,17,241&	3,4,12\\
\hline
60787&	89,683&	11,11\\
\hline
327685&	5,65537&	2,16\\
\hline
416179&	29,113,127&	14,14,7\\
\hline
549791&	11,151,331&	5,15,15\\
\hline
755915&	5,19,73,109&	2,9,9,18\\
\hline
1114129&	17,65537&	4,16\\
\hline
1472045&	5,37,73,109&	2,18,9,18\\
\hline
1549411&	31,151,331&	5,15,15\\
\hline
1912111&	31,61681&	5,20\\
\hline
2060863&	7,37,73,109&	3,18,9,18\\
\hline
3335735&	5,13,19,37,73&	2,6,9,18,9\\
\hline
6973057& 7,13,19,37,109&	3,6,9,18,18\\
\hline
\end{tabular}

\end{center}
}
\caption{Primitive numbers up to $5\times 10^6$, their prime decompositions and $o_4$ for the primes in the prime decomposition.}
\end{table}

In Theorem \ref{Gprop}, we give a criterion that ensures that a number $m$ is complete. It is based on the multiplicative group generated by the number 4 in $\bz_m$:
\begin{definition}\label{defgm}
Let $m$ be an odd natural number. We will denote by $\bz_m$ the finite ring of integers modulo $m$, $\bz/m\bz$. We use the notation $\bz_m^\times$ to indicate the multiplicative structure on $\bz_m$. We denote by $U(\bz_m)$ the set of elements in $\bz_m$ that have a multiplicative inverse. We denote by $G_m$ the group generated by 4 in $U(\bz_m)$,
$$G_m=\{4^j(\mod m) :j=0,1,\dots\}.$$
The order of $4$ in the group $U(\bz_m)$ is the smallest positive integer $a$ such that $4^a\equiv 1\mod m$. We denote $a$ by $o_4(m)$ and $o_4(m)=|G_m|$.

We denote by $\lcm(a_1,\dots,a_n)$ the lowest common multiple of the numbers $a_1,\dots,a_n$.
\end{definition}

Then, using this criterion, we prove in Theorem \ref{th3.3} that any prime power is a complete number.

The rest of the paper is devoted to the study necessary or sufficient conditions for a composite numbers to be primitive or complete. Section 3 contains several results 
in this direction; various conditions are given for a number to be complete or primitive based on the prime decomposition of the number and on the order of the number 4 in each of the multiplicative groups corresponding to these primes. Theorems \ref{th1.15} and \ref{cor1.14} give a sufficient condition for a number to be complete. Theorem \ref{cor1.14} also gives a condition for a number to be non-primitive. The key technical lemmas are Lemma \ref{lem1.16}, \ref{lem1.13} and \ref{lem2.14}.

In the last section of our paper, we illustrate the theory with some examples and we formulate some conjectures. 

\section{Prime powers}
We begin with some lemmas about the basic properties of extreme cycles. 
\begin{lemma}\label{lemi1.4}
If $x_0$ is an extreme cycle point then $x_0\in\bz$, $x_0$ has a periodic base 4 expansion 
\begin{equation}
x_0=\frac{a_0}4+\frac{a_1}{4^2}+\dots+\frac{a_{r-1}}{4^r}+\frac{a_0}{4^{r+1}}+\dots+\frac{a_{r-1}}{4^{2r}}+\dots,
\label{eqi1.4.1}
\end{equation}
with $a_k\in\{0,m\}$, 
and $0\leq x_0\leq\frac m3$. Hence
$$x_0=\frac{4^{r-1}a_0+4^{r-2}a_1+\dots+4a_{r-2}+a_{r-1}}{4^r-1}.$$

Moreover 
$$\{x_0 : x_0\mbox{ is an extreme cycle point }\}=X_L\cap\bz,$$
where $X_L$ is the attractor of the iterated function system 
$$\sigma_0(x)=\frac x4,\quad\sigma_m(x)=\frac{x+m}{4},$$
so 
$$X_L=\cup_{l\in\{0,m\}}\sigma_l(X_L),$$
\begin{equation}
X_L=\left\{\sum_{n=1}^\infty\frac{l_n}{4^n} : l_n\in\{0,m\} \mbox{ for all }n\in\bn\right\}.
\label{eqi.4.1}
\end{equation}

\end{lemma}

\begin{proof}
Let $l_0,\dots,l_{r-1}$ as in Definition \ref{defi1.2}. Then
$$x_0=\frac{x_{r-1}}4+\frac{l_{r-1}}{4}=\frac{x_{r-2}}{4^2}+\frac{l_{r-2}}{4^2}+\frac{l_{r-1}}{4}=\dots=\frac{x_0}{4^r}+\frac{l_0}{4^r}+\frac{l_1}{4^{r-1}}+\dots+\frac{l_{r-1}}{4}.$$ 
Iterating this equality to infinity we obtain the base 4 decomposition of $x_0$. Also 
$$0\leq x_0\leq \sum_{k=1}^\infty \frac{m}{4^k}=\frac{m}{3}.$$
From \eqref{eqi1.2.1}, using the triangle inequality we see that we must have $e^{2\pi i2x_0}=1$ so $x_0\in\bz/2$. If $x_0=(2m+1)/2$ with $m\in\bz$ then 
$x_1=(x_0+l_0)/4=\frac{2m+1+2l_0}{8}$, but since $2m+1+2l_0$ is odd it follows that $x_1\not\in\bz/2$. This contradicts the fact that $x_1$ is also an extreme cycle point so it satisfies \eqref{eqi1.4.1}. Thus $x_0\in\bz$.  

These statements show that $x_0$ is contained in $X_L\cap\bz$. Conversely, if $x_0\in X_L\cap \bz$ then, if $x_0\in \sigma_0(X_L)$, we have that there exists $x_{-1}\in X_L$ such that $x_0=\frac{x_{-1}}{4}$, and we get that $x_{-1}=4x_0\in \bz\cap X_L$. If $x_0\in\sigma_m(X_L)$ then there exists $x_{-1}\in X_L$ such that $x_0=\frac{x_{-1}+m}4$. Then $x_{-1}=4x_0-m\equiv x_0(\mod m)$. By induction, we obtain $x_{-1},x_{-2},\dots$ and digits $d_0,d_1,\dots...$ in $\{0,m\}$ such that $x_{-i}=\frac{x_{-i-1}+d_{i}}4$. Moreover, $x_0\equiv 4^ix_{-i}(\mod m)$. Since $4$ is mutually prime with $m$, it has a finite order $a$ in the multiplicative group of invertible elements in $U(\bz_m)$, so $4^a\equiv 1(\mod m)$. Then $x_0\equiv x_{-a}(\mod m)$. But since $x_0$ and $x_{-a}$ are contained in $X_L\subset [0,\frac m3]$, we get that $x_0=x_{-a}$ and thus $x_0$ is a point in an extreme cycle in $X_L\cap\bz$. 

\end{proof}

\begin{remark}\label{rem1.4}
Using Lemma \ref{lemi1.4}, one can develop an algorithm to determine the existence of non-trivial cycles. Take all the integers $k$ between $1$ and $m/3$. Define $x=k$. If 
$x\equiv 0\mod 4$ then set $x=x/4$. If $x+m\equiv0\mod4$ then set $x=(x+m)/4$. If none of these two conditions are satisfied then move to $k+1$. Do this as long as it is possible or until the point $x$  has already been checked before. If such a point is reached then stop; there is a non-trivial extreme cycle. If not, move on to the next integer $k+1$ and repeat these steps.

\end{remark}

\begin{theorem}\label{th1.10}
There are infinitely many primitive numbers.
\end{theorem}

\begin{proof}
Suppose there are only finitely many primitive numbers and let $m_1,\dots,m_s$ be all the primitive numbers strictly bigger than 3. Let $n$ be a common multiple for the numbers $o_4(9)$, $o_4(m_1)$, $\dots$, $o_4(m_s)$. Then 
$$4^{n+1}-1\equiv4-1= 3(\mod 9,\mod m_1,\dots,\mod m_s).$$
Let $m=\frac{4^{n+1}-1}{3}$. We have that $m$ is not divisible by 3, $m_1,\dots, m_s$, otherwise $4^{n+1}-1$ is divisible by $9,m_1,\dots,m_s$. So $m$ it is not divisible by any primitive number, therefore it must be complete. 

On the other hand, in Lemma \ref{lemi1.4}, let $r=n$, $a_{n-1}=a_{n-2}=a_{n-3}=m$, $a_0=\dots=a_{n-4}=0$. We have 
$$x_0=\frac{m(16+4+1)}{4^{n+1}-1}=\frac{\frac{4^{n+1}-1}{3}\cdot 21}{4^{n+1}-1}=7\in X_L\cap \bz.$$
Thus $x_0$ is a non-trivial extreme cycle point, so $m$ cannot be complete.
\end{proof}

\begin{lemma}\label{lem2.3}
Assume $m>3$ is odd and $x_j$ is an extreme cycle point for the digit set $\{0,m \}$. Then $x_j \equiv 0 (\mod 4)$ or $x_j \equiv -m (\mod 4)$.
\end{lemma}

\begin{proof}
We have
\beq
x_{j+1} = \frac{x_j + l_j}{4},
\eeq
where $l_j \in \{0,m \}$. Then
\beq
4 x_{j+1} = {x_j + l_j}.
\eeq
Considering the above modulo $4$, we have
\beq
0 \equiv x_j +m (\mod 4)
\eeq
or
\beq
0 \equiv x_j (\mod 4).
\eeq
\end{proof}

\begin{lemma}\label{lem2.4}
Let $m>3$ be an odd number not divisible by $3$ and $x_t$ be the largest extreme cycle point in the non-trivial extreme cycle $X$ for the digit set $\{0,m \}$. Then $x_t$ is divisible by $4$.
\end{lemma}

\begin{proof}
Assume for contradiction's sake that $x_t$ is odd. Then, with Lemma \ref{lem2.3}, the next cycle point is 
$$ \frac{x_t + m}{4} .$$
Since $x_t<m/3$ we get that 
$$ \frac{x_t + m}{4} > x_t . $$
This is a contradiction to the maximality of $x_t$.

Since $x_t$ is not odd, it is divisible by $4$ by the previous lemma.

\end{proof}

We mention also a way to determine if a coset of the group $G_m$ is an extreme cycle
\begin{proposition}\label{pr2.5}
Assume $m>3$ is odd. If a co-set $C$ of $G_m$ in $U(\bz_m)$ has the property that for all $x_j \in C$, $x_j < \frac{m}{2}$, then $C$ is an extreme cycle for the digit set $\{0,m \}$.
\end{proposition}

\begin{proof}
Let $C$ be such a co-set. Label the elements in $C$ such that $ x_j \equiv 4 x_{j+1} (\mod m)$, and if $a$ is the number of elements in $G_m$, $x_{a-1} \equiv 4 x_0  (\mod m)$. Then, since $0<x_{j+1} < \frac{m}{2}$, we have $0<4 x_{j+1} < 2m$, so
\beq
x_j = 4 x_{j+1} - km ,
\eeq
where $k \in \{0,1 \}$, and similarly for $x_0$ and $x_{a-1}$. Rearranging, we find that 
\beq
\frac{x_j + l_j}{4} =  x_{j+1}  ,
\eeq
where $l_j \in \{0,m\}$, and similarly for $x_0$ and $x_{a-1}$. Since $C$ contains only integers, by Lemma \ref{lemi1.4}, $C$ is an extreme cycle.
\end{proof}

\begin{proposition}\label{pri1.5}
Let $m$ and $k$ be some odd natural numbers. If $m$ is incomplete then $km$ is incomplete. 
\end{proposition}

\begin{proof}
If $m$ is incomplete, then by Theorem \ref{thi1.3}, there exists a non-trivial extreme cycle $\{x_0,\dots,x_{r-1}\}$ for the digits $\{0,m\}$. Multiplying the relations in Definition \ref{defi1.2} by $k$ we see that $\{kx_0,\dots,kx_{r-1}\}$ is a cycle for the digits $\{0,km\}$. With Lemma \ref{lemi1.4} we have that $x_i\in\bz$, so $kx_i\in\bz$ and therefore \eqref{eqi1.2.1} is satisfied for the points $kx_i$, and therefore we have a non-trivial extreme cycle for the digits $\{0,km\}$.
\end{proof}

\begin{theorem} \label{Gprop}
Let $m>3$ be an odd number not divisible by $3$. If any of the numbers $-1 (\mod m)$, $-2 (\mod m)$, $2 (\mod m)$, or $3(\mod m)$ is in $G_m$, then $m$ is complete. 
If $m>12$ and any of the numbers $5 (\mod m)$, $6 (\mod m)$, $7 (\mod m)$, $8(\mod m)$,  $9(\mod m)$,  $10(\mod m)$,  $11(\mod m)$ or $12(\mod m)$ is in $G_m$, then $m$ is complete. 
\end{theorem}

\begin{proof}

Assume for contradiction's sake that $m$ is incomplete. Then there is a non-trivial extreme cycle $X = \{x_0, ..., x_{r-1} \}$ for the digit set $\{0, m \}$. From the relation between the cycle points,
\beq
x_{j+1} = \frac{x_j+b_j}{4},
\eeq
where $b_j \in \{0,m\}$, we have that $4 x_{j+1} \equiv x_j (\mod m)$. Thus,
\beq
4^{r-k} x_{0} \equiv x_0 (\mod m, k\in\{0,\dots,r\}),
\eeq
so, for all $k\in\bn$, the number $4^kx_0$ is congruent modulo $m$ with an element of the extreme cycle $X$. But then, by the hypothesis, there is a number $c\in\{-1,2,-2,3\}$ in $G_m$. The number $cx_0$ is congruent modulo $m$ with an element in $X$, and since $x_0$ is arbitrary in the cycle, we get that $cx_j$ is congruent to an element in $X$ for any $j$.

 In the following arguments we use the fact that since ${m}$ is not divisible by $3$, the condition on cycle points $0 < x_j \leq \frac{{m}}{3}$ implies $0 < x_j < \frac{{m}}{3}$.

If $c=-1$, then $-x_0 (\mod {m}) \in X$. Since $0<x_0 < \frac{{m}}{3}$, $-x_0 (\mod {m})> \frac{{m}}{3}$, a contradiction.

If $c=-2$, then $-2x_0 (\mod {m}) \in X$. Since $0<x_0 < \frac{{m}}{3}$, $-2x_0 (\mod {m})> \frac{{m}}{3}$, a contradiction.

If $c=2$, then $2x_j (\mod {m}) \in X$ for all $j$. Let $x_N$ be the largest element of the extreme cycle. Since $0<x_N < \frac{{m}}{3}$, $2 x_N (\mod {m}) = 2x_N$. This number is in $X$, a contradiction to the maximality of $x_N$.

If $c=3$, then $3x_j (\mod {m}) \in X$ for all $j$. Let $x_N$ be the largest element of the extreme cycle. Since $0<x_N < \frac{{m}}{3}$, $3 x_N (\mod {m}) = 3x_N$. This number is in $X$, a contradiction to the maximality of $x_N$.

If $m>12$ then, as before, there is a number $c\in\{5,6,7,8,9,10,11,12\}$, such that the number $cx_0$ is congruent modulo $m$ with an element in $X$, and since $x_0$ is arbitrary in the cycle, we get that $cx_j$ is congruent to an element in $X$ for any $j$.

 In the following arguments we use the fact that since ${m}$ is not divisible by $3$, the condition on cycle points $0 \leq x_j \leq \frac{{m}}{3}$ implies $0 \leq x_j < \frac{{m}}{3}$. Let $x_t$ be the largest element in the extreme cycle. We have
 $$0< x_t < \frac{m}{3}. $$
By the Lemma \ref{lem2.4}, $x_t$ is divisible by four. Therefore, dividing by four, we get the next element in the extreme cycle, called $x_N$, and we have
$$  x_N < \frac{m}{12}. $$
For $c \in \{5,6,7,8,9,10,11,12\}$, $ x_t < c x_N < m $, so $cx_N(\mod m)=cx_N$ is a point in $X$ bigger than $x_t$, a contradiction to the maximality of $x_t$.

\end{proof}

\begin{corollary}\label{cor3.4}
For $n\geq 1$ the numbers $4^n+1$, $4^n-3$, $2\cdot4^n-1$ and $2\cdot 4^n+1$ are complete. For $n\geq 3$, the numbers $4^n-5,4^n-7,4^n-9,4^n-11, 2\cdot 4^n-3,2\cdot 4^n-5$ are complete.
\end{corollary}

\begin{proof}
If $m=4^n+1$ then $4^n=-1(\mod m)$. Then use Theorem \ref{Gprop}. Similarly for $4^n-3, 4^n-5,4^n-7,4^n-9,4^n-11 $.

If $m=2\cdot 4^n-1$, then $4^{n+1}-2=2(2\cdot 4^n-1)$ so $4^{n+1}=2(\mod m)$. Then use Theorem \ref{Gprop}. Similarly for $2\cdot 4^n+1,2\cdot 4^n-3,2\cdot 4^n-5$. 
\end{proof}

\begin{theorem}\label{th3.3}
If $p$ is a prime number, $p>3$ and $n\in\bn$, then $p^n$ is complete. 
\end{theorem}

\begin{proof}
It is well known (see e.g. \cite[page 45]{IrRo90}), that the equation $x^2\equiv b(\mod p^n)$ has 0 or two solutions. Let $a$ be the smallest positive integer such that $4^a\equiv 1(\mod p^n)$. If $a$ is even, then we have $(4^{a/2})^2\equiv 1(\mod p^n)$ so $4^{a/2}\equiv \pm1(\mod p^n)$. Since $4^{a/2}\neq 1(\mod p^n)$ we get $4^{a/2}\equiv -1(\mod p^n)$. 

If $a$ is odd, then $(4^{\frac{a+1}2})^2\equiv 4(\mod p^n)$. Therefore $4^{\frac{a+1}{2}}\equiv \pm 2(\mod p^n)$.

In both cases, the result follows from Theorem \ref{Gprop}
\end{proof}

\begin{remark}\label{rem3.3}
The proof of Theorem \ref{th3.3} indicates that it is enough to have exactly two solutions for both equations $x^2\equiv 1(\mod m)$ and $x^2\equiv 4(\mod m)$, to obtain that $m$ is complete. 
But the only odd numbers for which this condition holds are the prime powers. Indeed, if $m=p_1^{n_1}\dots p_r^{n_r}$, with $r\geq2$ and $n_1,\dots, n_r>0$, then, by the Chinese Remainder Theorem, there exists an integer $x$ such that $x\equiv -1(\mod p_1^{r_1})$, $x\equiv 1(\mod p_2^{r_2}),\dots x\equiv 1(\mod p_r^{n_r})$. This implies that $x^2\equiv 1(\mod p_k^{r_k})$ for all $k$, and therefore $x^2\equiv 1(\mod m)$. Also, it is clear that $x\neq \pm 1(\mod m)$. 
\end{remark}

\section{Composite numbers}

In this section we study composite numbers and we present some conditions for a number to be primitive or complete. We base our conditions on the prime decomposition of the numbers and on the order of the number 4 in the multiplicative group $U(\bz_m)$.

We begin with some properties of $o_4(m)$ that help in our computations. 

\begin{definition}\label{def1.11}
For a prime number $p\geq 3$, we denote by $\iota_4(p)$ the largest number $l$ such that $o_4(p^l)=o_4(p)$. We say that $p$ is {\it simple} if $o_4(p)<o_4(p^2)$, i.e., $\iota_4(p)=1$. 
\end{definition}

\begin{remark} 
The first non-simple prime number is 1093 and $o_4(1093)=o_4(1093^2)=182$. 
\end{remark}


\begin{proposition}\label{pr3.9}
Let $m$ and $n$ be mutually prime odd integers. Then 
$$o_4(mn)=\lcm(o_4(m),o_4(n)).$$
\end{proposition}

\begin{proof}
We have $a=o_4(mn)$ is the smallest integer such that $4^a\equiv 1(\mod mn)$. So $a$ is the smallest integer such that $4^a\equiv 1(\mod m)$ and $4^a\equiv 1(\mod n)$, which means that $a$ is the smallest integer that is divisible by $o_4(m)$ and $o_4(n)$ so it is the lowest common multiple of these two numbers. 
\end{proof}

\begin{proposition}\label{pr3.10}
Let $p$ be an odd prime number.  Then $o_4(p^k)=o_4(p)$ for $k\leq \iota_4(p)$ and $o_4(p^k)=p^{k-\iota_4(p)}o_4(m)$ for all $k\geq \iota_4(p)$. 
\end{proposition}

\begin{proof} For $k\leq \iota_4(p)$, the statement is trivial. 
Assume by induction that, for $k\geq \iota_4(p)$, $a_k:=o_4(p^k)=p^{k-\iota_4(p)}o_4(p)$ and $o_4(p^k)<o_4(p^{k+1})$. Then there exists $q$ not divisible by $p$ such that 
$4^{a_k}=1+qp^k$. Raise this to power $p$ using the binomial formula:
$$4^{pa_k}=1+p\cdot qp^k+q'p^{k+2},$$
for some integer $q'$. This implies that $a_{k+1}=o_4(p^{k+1})$ divides $pa_k$ and also that $pa_k$ is not $o_4(p^{k+2})$. Since $4^{a_{k+1}}\equiv 1(\mod p^{k+1})$ we have also $4^{a_{k+1}}\equiv 1(\mod p^k)$ so $a_k$ divides $a_{k+1}$. Thus $a_{k+1}$ is a number that divides $pa_k$ and is divisible by $a_k$, and by the induction hypothesis $a_{k+1}>a_k$. Thus $a_{k+1}=pa_k=p^{k+1-\iota_4(p)}o_4(p)$. Also, $o_4(p^{k+1})=pa_k\neq o_4(p^{k+2})$ so $o_4(p^{k+1})<o_4(p^{k+2})$. Using induction we obtain the result. 

\end{proof}

\begin{proposition}\label{pr3.14}
Let $p_1,\dots,p_r$ be distinct odd primes and $k_1,\dots,k_r\geq 0$. For $i\in\{1,\dots,r\}$, let $j_i\geq0$ be the largest integer such that $p_i^{j_i}$ divides $\lcm(o_4(p_1),\dots,o_r(p_r))$. Then 
\begin{equation}
o_4(p_1^{k_1}\dots p_r^{k_r})=\left(\prod_{i=1}^rp_i^{\max\{k_i-j_i-\iota_4(p_i),0\}}\right)\lcm(o_4(p_1),\dots,o_4(p_r)).
\label{eq1.14.1}
\end{equation}
\end{proposition}

\begin{proof}
With Propositions \ref{pr3.9} and \ref{pr3.10}, we have 
$$o_4(p_1^{k_1}\dots p_r^{k_r})=\lcm\left(p_i^{\max\{k_i-\iota_4(p_i),0\}}o_4(p_i); i\in\{1,\dots,r\}\right).$$
If $k_i-\iota_4(p_i)\leq j_i$, then $p_i^{\max\{k_i-\iota_4(p_i),0\}}$ already divides $\lcm(o_4(p_1),\dots,o_4(p_r))$ so it does not contribute to the right-hand side. If $k_i-\iota_4(p_i)> j_i$, then 
$p_i^{\max\{k_i-\iota_4(p_i),0\}}$ contributes with $p_i^{k_i-\iota_4(p_i)-j_i}$ to the right-hand side. Then \eqref{eq1.14.1} follows.

\end{proof}

The next proposition gives us some information about the structure of extreme cycles for primitive numbers.

\begin{proposition}\label{pr3.8}
Let $m$ be a primitive number and let $C=\{x_0,\dots,x_{p-1}\}$ be an extreme cycle. Then:
\begin{enumerate}
	\item The length $p$ of the cycle is equal to $o_4(m)$. 
	\item Every element of the cycle $x_i$ is mutually prime with $m$.
	\item The extreme cycle $C$ is a coset of the group $G_m$ in $U(\bz_m)$, $C=x_0G_m$.

\end{enumerate}
\end{proposition}

\begin{proof}
Suppose $x_0$ and $m$ have a common divisor $d>1$. Then, since $x_1=\frac{x_0+l_0}{4}$ we have that $4x_1$ is divisible by $d$ and since $d$ is odd it follows that $d$ divides $x_1$. By induction $d$ divides all elements of the cycle. But then $\{x_0/d,x_1/d,\dots,x_{p-1}/d\}$ is an extreme cycle for the digits $\{0,m/d\}$. But this contradicts the fact that $m$ is primitive. 

We have $4^jx_i\equiv x_{(i-j)(\mod p)}(\mod m)$ for all $i,j\in\{0,\dots,p-1\}$. Therefore $4^px_0\equiv x_0(\mod m)$. Since $x_0$ is in $U(\bz_m)$, we get that $4^p\equiv 1(\mod m)$, so $p$ divides $o_4(m)=:a$. Also, we have $x_0\equiv 4^ax_0\equiv x_{-a(\mod p)}(\mod m)$ so, since all the elements of the cycle are in $[0,m/3]$ we get that $x_0=x_{-a(\mod p)}$. Therefore $a$ is divisible by $p$. Thus $p=a=o_4(m)$. 

Since the length of the cycle is $o_4(m)$ which is the order of the group $G$, and since $4^jx_0(\mod m)=x_{-j(\mod p)}$, we get that $x_0G_m=C$.

\end{proof}

Together with Lemma \ref{lem1.13} and Lemma \ref{lem2.14}, the next lemma is the key technical point in our investigation. It allows us to verify completeness by induction. 
\begin{lemma}\label{lem1.16}
Let $a,b\geq 1$ be odd numbers. Assume that $o_4(ab) \geq \frac{2a+15}{12} o_4(b)$. Then $ab$ is not primitive.  
\end{lemma}

\begin{proof}
Suppose that $ab$ is primitive. Since $a>1$, $b$ is a proper divisor of $ab$ so $b$ is complete. By Proposition \ref{pr3.8}, there exists an extreme cycle $C$ and it is equal to a coset $x_0G_{ab}$ of the multiplicative group generated by 4 in $U(\bz_{ab})$. 
Consider the map $h:G_{ab}\rightarrow G_{b}$, $h(x)=x(\mod b)$. Then, $h$ is a homomorphism and it is onto. 
Let $|G_{ab}|=o_4(ab)=M o_4(b)=M |G_b|$, so that $h$ is an $M$-to-1 map, where $M\geq \frac{2a+15}{12}$.
Then the map $h':x_0G_{ab}\rightarrow (x_0(\mod b))G_b$, $h'(x_0x)=(x_0x)(\mod b)$, is also an $M$-to-1 map ($x_0$ is invertible in $\bz_{ab}^\times$, by Proposition \ref{pr3.8}, hence also in $\bz_b^\times$). 

So, in particular, there are exactly $M$ elements in $x_0G_{ab}$ which are mapped into $x_0(\mod b)$. These elements can be written $ x_0(\mod b) + kb (\mod ab)$ for $M$ different values of $k$, each in the set $\{0, \dots, a-1 \}$. Since $b$ is complete, by Proposition \ref{pr2.5}, the coset $(x_0(\mod b))G_b$ contains an element $> \frac{b}{2}$. Therefore we can assume $y_0:=x_0(\mod b)>\frac b2$. 

From Lemma \ref{lem2.3}, we know that the points in the cycle are congruent to 0 or $-ab$ modulo 4. So $y_0+kb\equiv 0$ or $-ab$ modulo 4, for all $M$ values of $k$ such that this point is in the extreme cycle. Since $b$ is odd, it has an inverse, $c$ in $\bz_4^\times$ and we have that $k\equiv -cy_0(\mod 4)$ or $k\equiv c(-ab-y_0)\mod 4$. Therefore the values of $k$ here belong to only two equivalence classes modulo 4, so in each set $\{4n,4n+1,4n+2,4n+3\}$ there are at most 2 values of $k$. Therefore, if we take the largest such $k$, if $M$ is even, then $k\geq 4(\frac M2-1)+1=2M-3$. If $M$ is odd, then the largest $k$ is at least $4(\frac{M-1}2-1)+4=2M-2$. So in both cases $k\geq 2M-3$. Then 
$$y_0+kb> \frac b2+(2M-3)b\geq \frac{ab}3,$$
and this contradicts the fact that an extreme cycle is contained in $[0,\frac{ab}3]$, by Lemma \ref{lemi1.4}.

\end{proof}

\begin{remark}\label{rem1.16}
We will use Lemmas \ref{lem1.16} and later Lemma \ref{lem2.14} to inductively prove that some numbers are complete: start with a prime power. We know these are complete, from Theorem \ref{th3.3}. Then, multiply by some number in such a way that one of the lemmas applies. Repeat this inductively. 
\end{remark}

The next result shows that, if we fix the prime numbers that appear in the decomposition, then we can check the completeness of all the numbers that have only these primes in the decomposition, by checking this property for the first finitely many such numbers. 
\begin{theorem}\label{th1.15}
Let $p_1,\dots,p_r$ be distinct odd primes. For $i\in\{1,\dots,r\}$, let $j_i\geq0$ be the largest number such that $p_i^{j_i}$ divides $\lcm(o_4(p_1),\dots,o_4(p_r))$. Assume that 
$p_1^{\iota_4(p_1)+j_1}\dots p_r^{\iota_4(p_r)+j_r}$ is complete.

Then $p_1^{k_1}\dots p_r^{k_r}$ is complete for any $k_1,\dots,k_r\geq0$.

\end{theorem}

\begin{proof}

Suppose there are some numbers $k_1,\dots,k_r\geq0$ such that $m=p_1^{k_1}\dots p_r^{k_r}$ is not complete. Therefore, a proper divisor of this number has to be primitive, relabeling the powers $k_i$, we can assume $m$ is primitive.  The hypothesis implies that for at least one $i$, $k_i\geq\iota_4(p_i)+j_i+1$. Relabeling again, we can assume $k_1\geq \iota_4(p_1)+j_1+1$. We have, with Proposition \ref{pr3.14}:
$$o_4(p_1^{k_1}\dots p_r^{k_r})=p_1^{k_1-\iota_4(p_1)-j_1}o_4(p_1^{\iota_4(p_1)+j_1}p_2^{k_2}\dots p_r^{k_r}).$$
Using Lemma \ref{lem1.16}, with $a=p_1^{k_1-\iota_4(p_1)-j_1}$, $b=p_1^{\iota_4(p_1)+j_1}p_2^{k_2}\dots p_r^{k_r}$, we get a contradiction.
\end{proof}

We performed a computer check to find all the primitive numbers less than $10^7$. The results are listed in Table \ref{tab1}. Using this and Theorem \ref{th1.15}, we get the next Corollary.
\begin{corollary}\label{cor1.12}
Let $p_1,\dots,p_r$ be distinct odd primes. For $i\in\{1,\dots,r\}$, let $j_i\geq0$ be the largest number such that $p_i^{j_i}$ divides $\lcm(o_4(p_1),\dots,o_4(p_r))$. Assume that $p_1^{\iota_4(p_1)+j_1}\dots p_r^{\iota_4(p_r)+j_r}<10^7$ and that the set $\{p_1,\dots,p_r\}$ does not contain any of the lists in the second column of Table \ref{tab1}. Then  $p_1^{k_1}\dots p_r^{k_r}$ is complete for any $k_1,\dots,k_r\geq0$.  
\end{corollary}

\begin{proof}

By Theorem \ref{th1.15}, it is enough to check that $m:=p_1^{\iota_4(p_1)+j_1}\dots p_r^{\iota_4(p_r)+j_r}$ is complete. If not, then it has to be divisible by some primitive number $m'$. Since $m<10^7$, we have that $m'<10^7$ so $m'$ has to be one of the numbers in Table \ref{tab1}. Then the list of primes in the prime decomposition of $m'$ is contained in the list of primes in the prime decomposition of $m$, and this contradicts the hypothesis. Therefore $m$ is complete. 
\end{proof}

\begin{lemma}\label{lem1.13}
The number of non-trivial cycle points for an odd number $m$ not divisible by $3$ is less than 
$$\min_n\left\{2^n \lceil{\frac{m}{3\cdot 4^n} \rceil}\right\}.$$
\end{lemma}

\begin{proof}

The phrasing in the statement of the lemma, "number of non-trivial cycle points,"  refers to the total number of points among all non-trivial cycles.

We know from Lemma \ref{lemi1.4} that the cycle points are contained in the intersection of the attractor $X_L$ with $\bz$. Also $X_L\subset[0,\frac m3]$. Therefore 
$$X_L\subset \bigcup_{a_0,a_1,\dots, a_{n-1}\in\{0,m\}}\sigma_{a_{n-1}}\dots\sigma_{a_0}\left[0,\frac m3\right]$$$$=
\bigcup_{a_0,a_1,\dots,a_{n-1}\in\{0,m\}}\left[\frac{a_0+4a_1+\dots 4^{n-1}a_{n-1}}{4^n},\frac{m}{3\cdot 4^n}+\frac{a_0+4a_1+\dots 4^{n-1}a_{n-1}}{4^n}\right].$$

The intervals in this union can be written as 
\beq\label{eq2.12}
\left[ \frac{m \sum_{k=0}^{n-1} l_k 4^k}{ 4^n} , \frac{m\left(1+ 3\sum_{k=0}^{n-1} l_k 4^k \right) }{3\cdot 4^n} \right] .
\eeq
with $l_0,\dots,l_{n-1}\in \{0,1\}$.

Because $m$ is not divisible by $3$ or $4$, the right endpoint is never an integer. Examining the left endpoint, we find
\beq
 \sum_{k=0}^{n-1} l_k 4^k < 4^n ,
\eeq
and thus, since $m$ is odd the left endpoint is an integer only if it is $0$. Since the only cycle containing $0$ is the trivial one, we have that the only non-trivial cycle points for $m$ are the interior points of the above intervals; there are $2^n$
such intervals at each iteration, and each one contains at most $\lceil \frac{m}{3\cdot 4^n} \rceil$ integers in its interior.
\end{proof}

\begin{lemma}\label{lem2.14}
Let $a,b\geq 1$ be odd numbers. Assume that $o_4(ab)>2^{\lceil\log_2\sqrt\frac{a}3\rceil}o_4(b)$. Then $ab$ is not primitive. 
\end{lemma}

\begin{proof}
We proceed as in the proof of Lemma \ref{lem1.13}. We take $n=\lceil \log_2\sqrt\frac{a}{3}\rceil$. Then $\frac{ab}{3\cdot 4^n}\leq b$, so the length of the intervals in \eqref{eq2.12} is at most $b$. As we have seen in the proof of Lemma \ref{lem1.13}. the endpoints of these intervals cannot be non-trivial cycle points. If $ab$ is primitive, then it has an extreme cycle $C$ which is a coset $x_0G_{ab}$, by Proposition \ref{pr3.8}.

Now, as in the proof of Lemma \ref{lem1.16}, define the map $h:x_0G_{ab}\rightarrow x_0G_b$, $x_0x\mapsto (x_0x)(\mod b)$. We saw that this is an $M$-to-1 map, with $M>2^n$. Therefore there are $M$ values of $k$ such that $x_0(\mod b)+kb$ is in the cycle $C$. However, the intervals in \eqref{eq2.12} contain at most one such point, since their length is $b$ and the endpoints are not extreme cycle points. We have $2^n<M$ such intervals, and this leads to a contradiction.

\end{proof}

\begin{remark}\label{rem2.15}
The estimate in Lemma \ref{lem2.14} is almost always better than the estimate in Lemma \ref{lem1.16}: we have $2^{\lceil\log_2\sqrt\frac{a}3\rceil}<\frac{2a+15}{12}$ for all odd numbers $a$ except $a=13$ and $a=15$, and for $a=15$, since $a$ is divisible by $3$ we know that $ab$ is not complete and not primitive. Despite this, we include this lemma since the arguments in the proof are different and they might be improved. 
\end{remark}

The next results show that if the order of $4$ in $U(\bz_m)$ is large, then $m$ cannot be primitive. 
\begin{theorem}\label{cor1.14}
Let $m$ be an odd number. Assume the following conditions are satisfied:
\begin{enumerate}
	\item For every proper divisor $d | m$, $d<m$, the number $d$ is complete.
	\item There exists $n\geq 0$ such that 
$$o_4(m)>\min_n\left\{2^n \lceil{\frac{m}{3\cdot 4^n} \rceil}\right\}.$$
\end{enumerate}
Then $m$ is complete. 

If only condition (ii) is satisfied, then $m$ is not primitive. 

Here $\lceil x\rceil $ is the smallest integer larger than or equal to $x$. 
\end{theorem}

\begin{proof}

If $m$ is primitive, then, by Proposition \ref{pr3.8}, there exists a cycle of length $o_4(m)$. The contradiction follows from Lemma \ref{lem1.13}. 
\end{proof}

\begin{corollary} \label{corRoot}
Let $m$ be an odd number. If
$$o_4(m)>2^{\lceil \log_2{\sqrt{\frac m3}}\rceil},$$
or in particular, if 
$$o_4(m)>\sqrt{\frac{4m}{3}}$$
then $m$ is not primitive. 
\end{corollary}

\begin{proof}
Let $n=\lceil \log_2{\sqrt{\frac m3}}\rceil$. Then $4^n\geq \frac m3$ so $\lceil \frac{m}{3\cdot 4^n} \rceil =1$. Furthermore, 
\beq
2^n \lceil \frac{m}{3 \cdot 4^n} \rceil = 2^n\leq 2^{\log_2\sqrt{\frac m3}+1}= \sqrt{ \frac{4m}{3} } .
\eeq
The rest follows from Theorem \ref{cor1.14}.

\end{proof}

\begin{corollary}\label{th1.12}
Let $p_1,\dots, p_r$ be distinct simple prime numbers strictly larger than 3. Assume the following conditions are satisfied:
\begin{enumerate}
	\item For any proper subset $F\subset\{1,\dots,r\}$ and any powers $k_i\geq 0$, $i\in F$, the number $\prod_{i\in F}p_i^{k_i}$ is complete. 
	\item None of the numbers $o_4(p_1),\dots, o_4(p_r)$ is divisible by any of the numbers $p_1,\dots,p_r$. 
	\item The following equation is satisfied:
	\begin{equation}
	\lcm(o_4(p_1),\dots, o_4(p_r))> 2^ {\lceil\log_2\sqrt{\frac{p_1 \dots p_r }{3}} \rceil}.
	\label{eq:}
	\end{equation}
\end{enumerate}
Then $p_1^{k_1}\dots p_r^{k_r}$ is complete.

\end{corollary}
\begin{proof}
Suppose there exists $k_1,\dots,k_r$ such that $p_1^{k_1}\dots p_r^{k_r}$ is not complete. Then pick $k_1,\dots,k_r$ such that $\sum_{i=1}^rk_i$ is as small as possible, with this property. Clearly, by (i) we can assume all $k_i\geq 1$. Then all the proper divisors of $p_1^{k_1}\dots p_r^{k_r}$ are complete. So $m:=p_1^{k_1}\dots p_r^{k_r}$ is primitive. 
By Propositions \ref{pr3.9} and \ref{pr3.10}, we have
$$o_4(m)=\lcm(o_4(p_1^{k_1},\dots, o_4(p_r^{k_r}))=\lcm(p_1^{k_1-1}o_4(p_1),\dots,p_r^{k_r-1}o_4(p_r))$$$$=p_1^{k_1-1}\dots p_r^{k_r-1}\lcm(o_4(p_1),\dots, o_4(p_r)).$$

From (iii), we get
$$p_1^{k_1-1}\dots p_r^{k_r-1}\lcm(o_4(p_1),\dots, o_4(p_r))>2^n \lceil{\frac{p_1 \dots p_r }{3\cdot 4^n} \rceil} p_1^{k_1-1}\dots p_r^{k_r-1}\geq 2^n \lceil{\frac{p_1^{k_1} \dots p_r^{k_r} }{3\cdot 4^n} \rceil}.$$
(we used the fact that for $a>0$, $N\in\bn$, $\lceil a\rceil N$ is an integer $\geq aN$, so it is bigger than $\lceil aN\rceil$). Since $m$ is primitive, Corollary \ref{corRoot} gives us a contradiction.
\end{proof}

\begin{corollary}\label{cor1.15}
Let $p_1,\dots, p_r$ be distinct simple prime numbers strictly larger than 3. Assume the following conditions are satisfied:
\begin{enumerate}
	\item None of the numbers $o_4(p_1),\dots, o_4(p_r)$ is divisible by any of the numbers $p_1,\dots,p_r$. 
	\item For any subset $\{i_1,\dots, i_s\}$ of $\{1,\dots,r\}$, with $s\geq 2$ the following inequality holds:
	\begin{equation}
	\lcm(o_4(p_{i_1}),\dots, o_4(p_{i_s}))>\sqrt{\frac43 {p_{i_1}\dots p_{i_s}}}.
	\label{eq1.15.1}
	\end{equation}
\end{enumerate}
Then the number $p_1^{k_1}\dots p_r^{k_r}$ is complete for any $k_1\geq0,\dots, k_r\geq0$. 
\end{corollary}

\begin{proof}
We proceed by induction on $r$. Theorem \ref{th3.3} shows that we have the result for $r=1$. Assume, the result holds for $r-1$ primes. Then the conditions (i),(ii) in Corollary \ref{th1.12} are satisfied and we check condition (iii). Let $m:=p_1\dots p_r$. 

 We have:
\beq
\sqrt{ \frac{4m}{3} } <o_4 (m) .
\eeq
Thus condition (iii) is satisfied and Corollary \ref{th1.12} gives us the result. 
\end{proof}

\begin{corollary}\label{cor2.15}
Let $p_1,\dots, p_r$ be distinct simple prime numbers strictly larger than 3. Assume the following conditions are satisfied:
\begin{enumerate}
	\item The numbers $o_4(p_1),\dots,o_4(p_r),p_1,\dots,p_r$ are mutually prime.
	\item $o_4 (p_j) > \sqrt{\sqrt{\frac43}p_j}$ for all $j$.
\end{enumerate}
Then the number $p_1^{k_1}\dots p_r^{k_r}$ is complete for any $k_1\geq0,\dots, k_r\geq0$. 
\end{corollary}
\begin{proof}
We use Corollary \ref{cor1.15}. For any subset $\{i_1,\dots,i_s\}$ of $\{1,\dots,r\}$ with $s\geq 2$ we have 
\beq
o_4 (p_{i_1}) \dots o_4 (p_{i_s}) > \sqrt{\sqrt{\frac43}p_{i_1}} \dots  \sqrt{\sqrt{\frac43}p_{i_s}} \geq \sqrt{\frac43 p_{i_1} \dots p_{i_s} } .
\eeq

\end{proof}

%
%
%
%
%

\begin{corollary}\label{cor1.24}
Let $a$ be a complete odd number. Let $p>3$ be a simple prime number. Assume that 
\begin{enumerate}
	\item $p$ does not divide $a$;
	\item $o_4(p)$ and $o_4(a)$ are mutually prime;
	\item $o_4(p)>2^{\lceil\log_2\sqrt\frac{p}{3}\rceil}$ (in particular if $o_4(p)=\frac{p-1}2$, $p>5$).

\end{enumerate}
Then $p^ka$ is complete for all $k\geq 0$.
\end{corollary}

\begin{proof}
Since $p$ does not divide $a$, $p^k$ is prime with $a$. With Propositions \ref{pr3.9}, \ref{pr3.10} we have
$$o_4(p^ka)=p^{k-1}o_4(p)o_4(b).$$
Also we have, for $k\geq 2$, since $p\geq 5$,
$$\lceil\log_2\sqrt\frac{p}{3}\rceil+\log_2p^{k-1}\geq \lceil\log_2\sqrt\frac{p}{3}\rceil+\log_2\sqrt{p^{k-1}}+1\geq \lceil\log_2\sqrt\frac{p}{3}\rceil+\lceil\log_2\sqrt{p^{k-1}}\rceil$$
$$\geq\lceil\log_2\sqrt\frac{p}{3}+\log_2\sqrt{p^{k-1}}\rceil=\lceil\log_2\sqrt\frac{p^k}{3}\rceil.$$

Therefore, 
$$p^{k-1}o_4(p)>2^{\lceil\log_2\sqrt\frac{p^k}{3}\rceil},$$
for $k\geq 2$ and also, from the hypothesis , for $k=1$. By Lemma \ref{lem2.14}, $p^ka$ cannot be primitive, for $k\geq 1$ and, because $a$ is complete and $p$ is prime, this means that $p^ka$ is complete. 

Note that $\frac{p-1}2\geq 2^{\lceil\log_2\sqrt\frac{p}{3}\rceil}$ for $p>5$, so this is indeed a peculiar case. 

\end{proof}

\begin{corollary}
Let $m$ be an odd number. If the index $x$ of $G_m$ in $U(\bz_m)$ satisfies $\frac{\phi(m)}{\sqrt{\frac43m}} > x$, where $\phi$ is Euler's totient function, then $m$ is not primitive.
\end{corollary}

\begin{proof}
We have $o_4(m) = |G_m|$ and $\phi(m) = |U(\bz_m)|$. Thus, from
$$
o_4 (m)=\frac{|U(\bz_m)|}{x}=\frac{\phi(m)}{x} > \sqrt{\frac43m}.
$$
The result follows from Corollary \ref{corRoot}.
\end{proof}

\section{Examples}

\begin{table}[ht]
{
\begin{tabular}{|c|c|c|c|c|c|c|c|c|c|c|c|c|c|}
\hline
$p$& $o_4(p)$&$p$& $o_4(p)$&$p$& $o_4(p)$&$p$& $o_4(p)$&$p$& $o_4(p)$&$p$& $o_4(p)$&$p$&$o_4(p)$\\
\hline
3&	1&	103&	51&	239&	119&	389&	194&	557&	278&	709&	354&	881&	55\\
\hline
5&	2&	107&	53&	241&	12&	397&	22&	563&	281&	719&	359&	883&	441\\
\hline
7&	3&	109&	18&	251&	25&	401&	100&	569&	142&	727&	121&	887&	443\\
\hline
11&	5&	113&	14&	257&	8&	409&	102&	571&	57&	733&	122&	907&	453\\
\hline
13&	6&	127&	7&	263&	131&	419&	209&	577&	72&	739&	123&	911&	91\\
\hline
17&	4&	131&	65&	269&	134&	421&	210&	587&	293&	743&	371&	919&	153\\
\hline
19&	9&	137&	34&	271&	135&	431&	43&	593&	74&	751&	375&	929&	232\\
\hline
23&	11&	139&	69&	277&	46&	433&	36&	599&	299&	757&	378&	937&	117\\
\hline
29&	14&	149&	74&	281&	35&	439&	73&	601&	25&	761&	190&	941&	470\\
\hline
31&	5&	151&	15&	283&	47&	443&	221&	607&	303&	769&	192&	947&	473\\
\hline
37&	18&	157&	26&	293&	146&	449&	112&	613&	306&	773&	386&	953&	34\\
\hline
41&	10&	163&	81&	307&	51&	457&	38&	617&	77&	787&	393&	967&	483\\
\hline
43&	7&	167&	83&	311&	155&	461&	230&	619&	309&	797&	398&	971&	97\\
\hline
47&	23&	173&	86&	313&	78&	463&	231&	631&	45&	809&	202&	977&	244\\
\hline
53&	26&	179&	89&	317&	158&	467&	233&	641&	32&	811&	135&	983&	491\\
\hline
59&	29&	181&	90&	331&	15&	479&	239&	643&	107&	821&	410&	991&	495\\
\hline
61&	30&	191&	95&	337&	21&	487&	243&	647&	323&	823&	411&	997&	166\\
\hline
67&	33&	193&	48&	347&	173&	491&	245&	653&	326&	827&	413&	1009&	252\\
\hline
71&	35&	197&	98&	349&	174&	499&	83&	659&	329&	829&	414&	1013&	46\\
\hline
73&	9&	199&	99&	353&	44&	503&	251&	661&	330&	839&	419&	1019&	509\\
\hline
79&	39&	211&	105&	359&	179&	509&	254&	673&	24&	853&	426&	1021&	170\\
\hline
83&	41&	223&	37&	367&	183&	521&	130&	677&	338&	857&	214&	1031&	515\\
\hline
89&	11&	227&	113&	373&	186&	523&	261&	683&	11&	859&	429&	1033&	129\\
\hline
97&	24&	229&	38&	379&	189&	541&	270&	691&	115&	863&	431&	1039&	519\\
\hline
101&	50&	233&	29&	383&	191&	547&	273&	701&	350&	877&	438&	1049&	131\\
\hline
\end{tabular}
}
\medskip

\caption{The primes less than $1049$ and their $o_4$}
\label{tab2}
\end{table}

\begin{example}\label{ex3.1}
We want to prove that $5^k\cdot 7^l$ is complete for any $k,l$. We have $o_4(5)=2$, $o_4(7)=3$ so 
$$\lcm(o_4(5),o_4(7))=6>2^{\lceil\log_2\sqrt{\frac {35}{3}}\rceil}=4.$$
Since $5$ and $7$ are simple primes, the result follows immediately from Corollary \ref{th1.12}.

\end{example}

\begin{example}
 Let us prove that $5^k\cdot 19^l$ is complete for any $k$, $l$. We have that $5^k$ is complete and $o_4(19)=9=\frac{19-1}{2}$ is prime with $o_4(5)=2$. So Corollary \ref{cor1.24} applies. The same argument applies to show that  $7^k\cdot11^l$, $5^k\cdot7^l\cdot23^m$ are complete. We can use this argument also for $7^k\cdot11^l\cdot 17^m$, but we have to start with $17^k$, since $o_4(17)=4$. Then $17^m\cdot 7^k$ is complete and $17^m\cdot 7^k\cdot 11^l$ is complete.  
\end{example}

\begin{example}\label{ex3.2}
Let us check that $5^k11^l$ is complete for any $k,l$. We have $o_4(5)=2$, $o_4(11)=5$. We have a small problem since $o_4(11)$ is divisible by 5, which is one of the primes. In Theorem \ref{th1.15} or Corollary \ref{cor1.12}, we have $\iota_4(5)=1$, $\iota_4(11)=1$, $\lcm(o_4(5),o_4(11))=10$, so $j_1$, the largest power of $5$ that divides the lcm 10, is 1, and $j_2=0$. So we have to check that $5^2\cdot11$ is complete, or that it does not contain any of the lists in the second column of Table \ref{tab1}. And that is clear. 

We could also try to use Theorem \ref{cor1.14} or Corollary \ref{corRoot}. For that, since we know that $5$ and $11$ are complete (because they are prime), we have to check that $5\cdot 11$ and $5^2\cdot 11$ are not primitive. We can use Corollary \ref{corRoot} to check that $5\cdot 11$ is complete 
$$o_4(5\cdot 11)=\lcm(o_4(5),o_4(11))=10>2^{\lceil\log_2\sqrt{\frac{5\cdot 11}{3}}\rceil}=8.$$
However, we cannot use this for $5^2 11$, because
$$o_4(5^2\cdot11)=10<2^{\lceil\log_2\sqrt{\frac{5^2 11}{3}}\rceil}=16.$$
The minimum in Theorem \ref{cor1.14} gives the same value, 16. 
\end{example}

Looking at Table \ref{tab1}, we formulate the following conjecture:
\begin{conjecture}\label{con1}
Let $m$ be a primitive number. Then
\begin{enumerate}
	\item $m$ is square-free.
	\item If $m=p_1\dots p_r$ is the prime decomposition of $m$, then there exists $i$ such that 
	$$\lcm(o_4(p_1),\dots,o_4(p_r))=o_4(p_i).$$
\end{enumerate}
\end{conjecture}

A weaker conjecture is the following: 

\begin{conjecture}\label{con2}
Let $m$ be an odd number not divisible by 3 and let $m=p_1^{k_1}\dots p_r^{k_r}$ be its prime decomposition. If the numbers $o_4(p_1),\dots,o_4(p_r),p_1,\dots,p_r$ are mutually prime then $m$ is complete. 

\end{conjecture}

It is easy to see that Conjecture \ref{con1} implies Conjecture \ref{con2}, for if $m$ is not complete, then it is divisible by some primitive number, and by Conjecture \ref{con1}, the orders cannot be mutually prime.

 \begin{acknowledgements}
This work was partially supported by a grant from the Simons Foundation (\#228539 to Dorin Dutkay).
\end{acknowledgements}

\end{document}